\documentclass[12pt]{article}
\usepackage{amsmath}
\usepackage{amssymb}
\usepackage{amsthm}
\usepackage{epic,eepic}
\usepackage{float}
\usepackage{enumerate}
\usepackage{hyperref}%
\newtheorem{theorem}{Theorem}[section]
\newtheorem{proposition}[theorem]{Proposition}
\newtheorem{corollary}[theorem]{Corollary}
\newtheorem{fact}[theorem]{Fact}
\newtheorem{conjecture}[theorem]{Conjecture}
\newtheorem{observation}[theorem]{Observation}
\newtheorem{falsetheorem}[theorem]{\lq{False Theorem}\rq}
\theoremstyle{definition}
\newtheorem{definition}[theorem]{Definition}
\newtheorem{example}[theorem]{Example}
\newtheorem*{example*}{Example}
\theoremstyle{remark}
\newtheorem{remark}[theorem]{Remark}
\newtheorem*{remarkwn}{Remark}
\numberwithin{equation}{section}
\newcommand{\Cp}{C_+}

\newcommand{\yato}{{\downarrow{\kern-7pt{\raise8pt\hbox{$>$}}}}}
\title{Branching problems of Zuckerman derived functor modules
\\[2ex]
\textit{\normalsize 
Dedicated to Gregg Zuckerman on the occasion of his 60th birthday}
}
\author{Toshiyuki Kobayashi\footnote{Partially supported by
Institut des Hautes \'{E}tudes Scientifiques, France and
        Grant-in-Aid for Scientific Research (B) (22340026), Japan
        Society for the Promotion of Science}}
\date{}
\begin{document}
\maketitle

\begin{abstract}
We discuss recent developments on branching problems of irreducible
unitary representations $\pi$ of real reductive groups when
restricted to reductive subgroups.
Highlighting the case where the underlying $(\mathfrak{g},K)$-modules
of $\pi$ are isomorphic to Zuckerman derived functor modules
$A_{\mathfrak{q}}(\lambda)$, 
we show various and rich features of branching laws 
 such as infinite multiplicities,
irreducible restrictions, multiplicity-free restrictions,
and discrete decomposable restrictions.
We also formulate a number of conjectures.
\end{abstract}



\section{Introduction}
Zuckerman derived functor is powerful algebraic machinery
 to construct irreducible unitary representations
 by cohomological parabolic induction.  
The $({\mathfrak {g}},K)$-modules $A_{\mathfrak{q}}(\lambda)$, 
 referred to as Zuckerman derived functor modules,
 give a far reaching generalization
 of the Borel--Weil--Bott construction
 of irreducible finite dimensional representations
 of compact Lie groups.  
They include Harish-Chandra's discrete series representations
 of real reductive Lie groups
 as a special case,
 and may be thought of
 as a geometric quantization of elliptic orbits
 (see Fact \ref{fact:ZDF}).  

{\textit{Branching problems}}
 in representation theory ask 
 how irreducible representations
 $\pi$ of a group $G$ decompose when restricted
 to a subgroup $G'$.

The subject of our study is branching problems 
 with emphasis on the setting when
 $(G,G')$ is a reductive symmetric pair
 (Subsection \ref{subsec:symm}),
 and when $\pi$ is the unitarization of a
Zuckerman derived functor module $A_{\mathfrak{q}}(\lambda)$.
We  see 
 that branching problems in this setting include a wide
range of examples:
a very special case is equivalent to finding the
Plancherel formula for homogeneous spaces 
(e.g. Proposition \ref{prop:Mackey} and Example~\ref{ex:GL})
 and another special case is of
combinatorial nature
(e.g.\ the Blattner formula).  

In this article, we give new perspectives 
 on branching problems
 by revealing the following surprisingly rich and various features:
\begin{itemize}
\item[$\bullet$]
The multiplicities may be infinite 
(Section \ref{sec:wild})
 and may be one
(Section \ref{subsec:mfconj}).  
\item[$\bullet$]
The restriction may stay irreducible
 (Section \ref{sec:aibl}).  
\item[$\bullet$]
The spectrum may be
 purely continuous and may be discretely
decomposable
 (Section \ref{sec:deco}).
\end{itemize}

Finally,
 we present a number of open problems
 that might be interesting for further study
 (see Conjectures~\ref{conj:mf}, \ref{conj:bdd}, \ref{conj:deco}, and \ref{conj:AV}).  

This article is based on the talk
 presented at the conference
\lq\lq Representation Theory and Mathematical Physics\rq\rq\
 in honor of Gregg Zuckerman's 60th birthday
 at Yale University on October 2009.  
The author is one of those who have been inspired by
Zuckerman's work, and would like to express his sincere gratitude to the organizers
 of the stimulating conference,
Professors J. Adams, M. Kapranov, B. Lian, and S. Sahi
for their hospitality.

\section{Wild aspects of branching laws}
\label{sec:wild}

\subsection{Analysis and synthesis}
\label{sec:irreddeco}

One of the most distinguished feature of \textit{unitary} representations 
is that they are always built up from the smallest objects, namely,
irreducible ones.
For a locally compact group $G$, 
 we denote by $\widehat G$
 the set of equivalence classes of irreducible
unitary representations of $G$, 
 endowed with the Fell topology.

\begin{fact}
[Mautner--Teleman]
\label{fact:irreddeco}
Every unitary representation $\pi$ of a locally compact group $G$
is unitarily equivalent to a
direct integral of irreducible unitary representations:
\begin{equation}\label{eqn:3.1.1}
\pi \simeq \int_{\widehat{G}}^{\oplus} n_{\pi}(\sigma) \sigma
                  \, d\mu(\sigma).
\end{equation}
Here, $d\mu$ is a Borel measure on $\widehat{G}$,
$n_{\pi} : \widehat{G} \to \mathbb{N} \cup \{ \infty \}$
is a measurable function,
and
$n_\pi (\sigma) \sigma$ stands for the multiple of an
irreducible unitary representation $\sigma$ with multiplicity
$n_\pi(\sigma)$.
\end{fact}
The decomposition \eqref{eqn:3.1.1} is unique if $G$ is
of type~I
in the sense of von Neumann algebras.
Reductive Lie groups are of type I.
Then the \textit{multiplicity function} $n_\pi$ is well-defined up to
a measure zero set with respect to $d\mu$.
We say
 that $\pi$ has a \textit{uniformly bounded multiplicity}
 if there is $C>0$
 such that $n_\pi(\sigma)\le C$ almost everywhere;
 $\pi$ is \textit{multiplicity-free} if $n_\pi(\sigma)\le 1$
almost everywhere,
or equivalently,
if the ring of continuous $G$-endomorphisms of $\pi$ is commutative.

\subsection{Branching laws and Plancherel formulas}

Suppose that $G'$ is a closed subgroup of $G$.
Here are two basic settings
 where the problem of decomposing unitary representations
 arises naturally.
\begin{itemize}
\item[1)]
(Induction $G' \uparrow G$)
\textit{Plancherel formula}.

For simplicity,
assume that there exists a $G$-invariant Borel measure on the
homogeneous space $G/G'$.
Then the group $G$ acts unitarily on the Hilbert space
$L^2(G/G')$ by translations.
The irreducible
decomposition of the regular representation of $G$ on
$L^2(G/G')$
is called the {\textit{Plancherel formula}} for $G/G'$.  
\item[2)]
(Restriction $G \downarrow G'$)
\textit{Branching laws}.

Given an irreducible unitary representation $\pi$ of $G$.
By the symbol $\pi|_{G'}$,
we think of $\pi$ as a representation of the subgroup $G'$.
The {\textit{branching law}} of the restriction $\pi|_{G'}$ means the formula of
 decomposing $\pi$ into irreducible representations of $G'$.
Special cases of branching laws include 
 the classical Clebsch--Gordan formula,
 or more generally,
the decomposition of the tensor product of two irreducible
representations
(\textit{fusion rule}),
and the Blattner formula,
 etc.  
\end{itemize}

\subsection{Symmetric pairs}
\label{subsec:symm}

We are particularly interested in the branching laws
 with respect to reductive symmetric pairs.  
Let us fix some notation.  

Suppose $\sigma$ is an involutive automorphism of a Lie group $G$.
We denote by
$G^\sigma := \{ g \in G: \sigma g = g \}$,
the group of fixed points by $\sigma$.
We say that $(G,G')$ is a \textit{symmetric pair} if $G'$ is an open
subgroup of $G^\sigma$.
Then the homogeneous space $G/G'$ becomes an affine symmetric space with
respect to the canonical $G$-invariant affine connection.
The pair $(G,G')$ is said to be a \textit{reductive symmetric pair} 
if $G$ is reductive.  
Further,
 if $G'$ is compact
 then $G/G'$ becomes a Riemannian symmetric space.  

\begin{example}\label{ex:symm}
1)
(group case)\enspace
Let $\grave{}G$ be a Lie group,
$G := \grave{}G \times \grave{}G$
the direct product group,
and $\sigma \in \operatorname{Aut}(G)$ be defined as 
$\sigma(x,y) := (y,x)$.
Then
$G^\sigma \equiv \operatorname{diag}(\,\grave{}G) 
 := \{ (x,x): x \in \grave{}G \}$.
Since the homogeneous space $G/G^\sigma$ is diffeomorphic to
$\grave{}G$, 
we refer to the symmetric pair
$(G,G^\sigma) = (\,\grave{}G \times \grave{}G, \operatorname{diag}(\,\grave{}G))$
as a \textit{group case}.

2)\enspace
The followings are chains of reductive symmetric pairs:
\begin{align*}
&GL(2n,\mathbb{H}) \supset GL(n,\mathbb{C}) \supset 
 GL(n,\mathbb{R}) \supset GL(p,\mathbb{R}) \times GL(q,\mathbb{R})
\quad (p+q=n), 
\\
&O(4p,4q) \supset U(2p,2q) \supset Sp(p,q) \supset U(p,q)
 \supset O(p,q).  
\end{align*}
\end{example}

\subsection{Finite multiplicity theorem of van den Ban}

Let $(G, G')$ be a reductive symmetric pair. 

The irreducible decomposition \eqref{eqn:3.1.1} is well-behaved
 for the induction $G'\uparrow G$, 
namely, 
for the Plancherel formula of the symmetric space $G/G'$:
\begin{fact}[van den Ban \cite{xban}]
\label{fact:Ban}
Suppose $(G,G')$ is a reductive symmetric pair.
Then the regular representation $\pi$ on 
$L^2(G/G')$ has a uniformly bounded multiplicity.
\end{fact}

\subsection{Plancherel formulas v.s. branching laws}
\label{subsec:Mackey}

Fairly many cases of the Plancherel formula for $L^2(G/G')$
 treated in Fact \ref{fact:Ban}
 can be realized as a special example
 of branching laws
 of the restriction of irreducible unitary representations of other groups.  
For example,
 we recall from \cite[Propositions 6.1, 6.2]{xkdecoalg}
 and \cite[Theorem 36]{xrims40}:
\begin{proposition}
\label{prop:Mackey}
Let $G/G'$ be a reductive symmetric space.  
Then the regular representation of $G$
 on $L^2(G/G')$ is unitarily equivalent to 
the restriction $\pi|_G$
 for some irreducible unitary representation $\pi$
 of a reductive group $\widetilde G$  containing $G$ as its subgroup
 if $(G,G')$ fulfills one of the following conditions:
\begin{itemize}
\item[\upshape(A)]
$G'$ is compact and the crown domain $D$
 of the Riemannian symmetric space $G/G'$
 is a Hermitian symmetric space,

or

\item[\upshape(B)]
$G'/Z_G$ has a split center.  
Here $Z_G$ stands for the center of $G$.
\end{itemize}
\end{proposition}
\begin{remark}\label{rem:Pb}
\begin{enumerate}
\item[1)]
Most Riemannian symmetric 
 pairs $(G,G')$ satisfy
 the assumption (A) 
 (see \cite{xkrst} for details).  

\item[2)]
As the proof below shows,
\[
  \widetilde G \supset G \supset G'
\]
is a chain of reductive symmetric pairs.  

\item[3)]
We can take $\pi$ to be the unitarization of some
$A_{\mathfrak{q}}(\lambda)$ in (A) and also in (B) when $G$ is a
complex reductive Lie group.

\item[4)]
There are some more cases other than (A) or (B) for which
the conclusion of Proposition~\ref{prop:Mackey} holds.
For instance, see Example~\ref{ex:GL} 
  for the group case $L^2(GL(n, \mathbb C))$
  and also for a more general case
  $L^2(GL(2n, \mathbb R)/GL(n, \mathbb C))$.

.
\end{enumerate}
\end{remark}

\begin{proof}
[Outline of the proof]
The choice of $\pi$ and 
 $\widetilde G$ depends on each case (A) and (B).  

(A) We take $\widetilde G$ to be the automorphism group
 of $D$,
and $\pi$ to be any holomorphic discrete series representation
 of $\widetilde G$ of scalar type.  
Then $\pi$ is realized in the Hilbert space
 consisting of square integrable,
 holomorphic sections of a $G$-equivariant holomorphic line bundle over
$D$.  
Since holomorphic sections are determined uniquely
 by the restriction to the totally real submanifold $G/G'$,  
 we get a realization of the restriction 
 $\pi|_G$ in a certain Hilbert subspace of 
 ${\mathcal {A}}(G/G')$, 
 which itself is not $L^2(G/G')$
 but is unitarily equivalent to the regular representation 
 on $L^2(G/G')$
 (see \cite{xhower}).  

(B) Let $P$ be a maximal parabolic subgroup of $G$
 whose Levi part is $G'$.  
Take $\widetilde G$ to be the direct product $G \times G$,
 and $\pi$ to be the outer tensor product representation
 $\pi_1 \boxtimes \pi_2$
 where $\pi_1$ is a degenerate unitary principal series
representation induced from a unitary character of $P$
 and $\pi_2$ is the contragredient representation of $\pi_1$.  
Then apply the Mackey theory.  
\end{proof}

\begin{example}
\label{ex:Mackey}
1)
The regular representation  on
 $L^2(G/G')=L^2(GL(n,{\mathbb{R}})/O(n))$
 is unitarily equivalent to the restriction
 of a holomorphic discrete series representation
 of
 $\widetilde G :=Sp(n,{\mathbb{R}})$ to $G$.

2)
The regular representation on
 $L^2(GL(n,{\mathbb{R}})/GL(p,{\mathbb{R}}) \times GL(q,{\mathbb{R}}))$
 with $(p+q=n)$
 is unitarily equivalent to the restriction of
a degenerate principal representation of
$\widetilde G:=GL(n,{\mathbb{R}}) \times GL(n,{\mathbb{R}})$
to $G$ (namely, to the tensor product representation).  
\end{example}

\subsection{Wild aspects of branching laws}
\label{subsec:2.5}

Retain our assumption
 that $(G,G')$ is a reductive symmetric pair.

Proposition \ref{prop:Mackey} suggests that branching problems
 include a wide range of examples.
In fact,
 while the `good behavior' in Fact \ref{fact:Ban}
 for the Plancherel formula of the symmetric space $G/G'$,
the branching law of the restriction $\pi|_{G'}$ does not behave well
in general.
Even when $\pi_K$ is a Zuckerman derived functor module
 $A_{\mathfrak {q}}(\lambda)$, 
 we cannot expect:
\begin{falsetheorem}
\label{thm:false}
Let $(G,G')$ be a reductive symmetric pair,
 and $\pi$ an irreducible unitary representation of $G$.  
Then the multiplicities 
 of the discrete spectrum
 in the branching laws $\pi|_{G'}$ are finite.  
\end{falsetheorem}

\begin{remark}
Such a multiplicity theorem holds for reductive symmetric pairs 
 $(G, G')$
under the assumption that the restriction $\pi|_{G'}$ is infinitesimally 
 discretely decomposable 
 in the sense of Definition \ref{def:decores}
(cf. \cite{xkdecoass, deco-euro}).
A key to the proof is 
 Theorem \ref{thm:sufadm} on a criterion 
of $K'$-admissibility and 
Corollary \ref{cor:necdeco}
 on an estimate of the associated variety. 
See Remark \ref{rem:decosymm} 
 for the case $\pi_K \simeq A_{\mathfrak {q}}(\lambda)$.    
\end{remark}
Before giving a counterexample to (false) `Theorem' \ref{thm:false}
 about the discrete spectrum,
we discuss an easier case, namely,
 an example of infinite multiplicities
 in the continuous spectrum
 of the branching law: 
\begin{proposition}
[$G \times G \downarrow \operatorname{diag}G$ ]
\label{prop:minf}
{\upshape{(Gelfand--Graev \cite{xgg}.)}}
If $\pi_1$ and $\pi_2$ are two unitary
 principal series representations 
 of\/ $G = SL(n,\mathbb{C})$ ($n\ge3$), 
 then the multiplicities in the decomposition of the tensor product 
$\pi_1 \otimes \pi_2$ are infinite almost everywhere with respect to
the measure $d\mu$ in the direct integral \eqref{eqn:3.1.1}.
\end{proposition}
We recall
 the underlying $({\mathfrak {g}},K)$-modules
 of unitary principal
series representations of a complex reductive Lie group are obtained  
 as a special case of Zuckerman derived
functor modules $A_{\mathfrak {q}}(\lambda)$.

Hence we get
\begin{observation}\label{ob:infc}
The multiplicities of the continuous spectrum in the branching law of
the restriction $\pi|_{G'}$ may be infinite even in the setting where
$\pi_K\simeq A_{\mathfrak{q}}(\lambda)$ and $(G,G')$ is
 a reductive symmetric pair.
\end{observation}

Here is a more delicate example,  
 which yields a counterexample to 
 (false) \lq{Theorem}\rq
 \ref{thm:false}
 about the discrete spectrum.
\begin{proposition}
[$G_{\mathbb{C}} \downarrow G_{\mathbb{R}}$]
\label{prop:mfinf}
{\upshape{(see \cite{xkdecoaspm})}}
There exist an irreducible unitary principal series representation
 $\pi$ 
of\/ $G=SO(5,{\mathbb{C}})$
 and two irreducible unitary representations $\tau_1$ 
(a holomorphic discrete series representation) 
and $\tau_2$
(a non-holomorphic discrete series representation) 
 of the
 subgroup $G'=SO(3,2)$
such that
$$
0 < \dim \operatorname{Hom}_{G'}
 (\tau_1, \pi|_{G'}) < \infty
\quad\text{and}\quad 
\dim \operatorname{Hom}_{G'}
 (\tau_2, \pi|_{G'}) = \infty.
$$
Here, 
$\operatorname{Hom}_{G'}(\cdot, \cdot)$ denotes the space of
continuous $G'$-intertwining operators.
\end{proposition}

\section{Almost irreducible branching laws}
\label{sec:aibl}

Let $G$ be a real reductive Lie group,
 $G'$ a subgroup, 
 and $\pi$ an irreducible unitary representation of $G$.  

We have seen some wild aspects of branching laws
 in the previous section.    
As its opposite extremal case, 
this section highlights especially nice cases,
namely, 
where the restriction $\pi|_{G'}$ remains irreducible
 or almost irreducible
in the following (obvious) sense:
\begin{definition}\label{def:ai}
We say a unitary representation $\pi$ is \emph{almost irreducible}
if $\pi$ is a finite direct sum of irreducible representations.
\end{definition}

It may well happen that the restriction $\pi|_{G'}$ is almost
irreducible when $G'$ is a maximal parabolic subgroup of $G$,
but is a rare phenomenon
 when $G'$ is a reductive subgroup.  
Nevertheless, 
we find in Subsections \ref{subsec:3.2}--\ref{subsec:3.4} that
there exist a small number of examples where the restriction
$\pi|_{G'}$ stays irreducible,
 or is almost irreducible
 in some cases.

We divide such irreducible unitary representations $\pi$ of $G$ into three cases,
according as $\pi_K$ are Zuckerman derived functor modules
 $A_{\mathfrak {q}}(\lambda)$
(see Theorem \ref{thm:Aqai}),
 principal series representations
(see Theorem \ref{thm:aips}),
and minimal representations 
(see Theorem \ref{thm:ain}).
{}From the view point
 of the Kostant--Kirillov--Duflo orbit method, 
 they
 may be thought of as the geometric quantization
 of elliptic,
 hyperbolic,
 and nilpotent orbits, 
 respectively.  

\subsection{Restriction to compact subgroups}

First of all,
we observe that almost irreducible restrictions $\pi|_{G'}$
 happen only when $G'$ is non-compact
 if $\dim \pi = \infty$.

Let $K$ be a maximal compact subgroup of a real reductive Lie group
$G$. 

\begin{observation}\label{prop:GK}
 For any irreducible
infinite dimensional unitary representation $\pi$ of $G$,
the branching law
  of the restriction $\pi|_K$ contains
  infinitely many irreducible representations
of $K$.
\end{observation}

\begin{proof} 
Clear from Harish-Chandra's admissibility
theorem (see Fact~\ref{fact:HCadm} below).
\end{proof}

For later purpose,
 we introduce the following terminology: 
\begin{definition}
\label{def:cptadm}
Suppose $K'$ is a compact group 
 and $\pi$ is a representation of $K'$.  
We say $\pi$ is $K'$-admissible 
 if 
$
     \dim \operatorname{Hom}_{K'}(\tau, \pi) < \infty
    \quad\text{for any }
    \tau \in \widehat{K'}.  
$
\end{definition}
With this terminology,
 we state:
\begin{fact}[Harish-Chandra's admissibility theorem]
\label{fact:HCadm}
Any irreducible unitary representation $\pi$ of $G$
 is $K$-admissible.  
\end{fact}

We shall apply the notion of $K'$-admissibility
 when $K'$ is a subgroup of $K$,
 and see that it plays a crucial role
 in the theory of discretely decomposable restrictions 
in Section \ref{sec:deco}.

\subsection{Irreducible restriction $\pi|_{G'}$ with
$\pi_K=A_{\mathfrak{q}}(\lambda)$}
\label{subsec:3.2}

This subsection discusses for which triple $(G,G',\pi)$
 the restriction $\pi|_{G'}$ is (almost)
irreducible in the setting that the underlying
 $(\mathfrak{g},K)$-module $\pi_K$ is
 isomorphic to a
 Zuckerman derived functor module
$A_{\mathfrak{q}}(\lambda)$. 

Let $\mathfrak{q}$ be a $\theta$-stable
parabolic subalgebra of
$\mathfrak{g}_{\mathbb{C}}=\mathfrak{g}\otimes_{\mathbb{R}}\mathbb{C}$,
$L := N_G(\mathfrak{q}) \equiv 
\{g\in G: \operatorname{Ad}(g)\mathfrak{q} = \mathfrak{q}\}$,
and $\overline{A_{\mathfrak {q}}(\lambda)}$
 the unitary representation of $G$ whose
underlying $(\mathfrak{g},K)$-module is $A_{\mathfrak{q}}(\lambda)$.

\begin{theorem}
[\cite{xkvoc}]
\label{thm:Aqai}
Suppose that $(G,G',L)$ is one of the following triples:
\[
\begin{matrix}
  G &&G' &L
  \\
  \hline
 SU(n,n)      &&Sp(n,\mathbb{R}) &U(n-1,n)
  \\
 SU(2p,2q)    &&Sp(p,q)          &U(2p-1,2q)
  \\
 SO_0(2p,2q)  &&SO_0(2p,2q-1)    &U(p,q)
  \\
 SO_0(4,3)    &&G_2(\mathbb{R})  &SO_0(4,1) \times SO(2)
  \\
 SO_0(4,3)    &&G_2(\mathbb{R})  &SO(2) \times SO_0(2,3)
\\
\hline
SL(2n,\mathbb{C})  && Sp(n,\mathbb{C})  & GL(2n-1,\mathbb{C})
\\
SO(2n,\mathbb{C})  && SO(2n-1,\mathbb{C})  & GL(n,\mathbb{C})
\\
SO(7,\mathbb{C})  && G_2(\mathbb{C})  & \mathbb{C}^\times \times SO(5,\mathbb{C})
\\
\hline
SU(2n)  && Sp(n)  & U(2n-1)
\\
SO(2n)  && SO(2n-1)  & U(n)
\\
SO(7)  && G_{2, compact}  & SO(2) \times SO(5)
\end{matrix}
\]
Then, the restriction
$\overline{A_{\mathfrak {q}}(\lambda)}|_{G'}$
 is almost irreducible for
any $\lambda$ 
satisfying the positivity and integrality condition
(see Subsection \ref{sub:7.1.3}).
Further, 
the restriction $\overline{A_{\mathfrak {q}}(\lambda)}|_{G'}$
 stays irreducible
 if the character $\lambda|_{\mathfrak {l}'}$ 
 of ${\mathfrak {l}}' := {\mathfrak {g}}' \cap {\mathfrak {l}}$
 is
 in the good range with respect to 
$\mathfrak{q}':=\mathfrak{g}'_{\mathbb{C}}\cap\mathfrak{q}$
 (see \eqref{eqn:good}).
On the level of Harish-Chandra modules,
we have an isomorphism
\[
A_{\mathfrak{q}}(\lambda)
\simeq
A_{\mathfrak{q}'}(\lambda|_{\mathfrak{l}'}),
\]
as $(\mathfrak{g}',K')$-modules.  
\end{theorem}

\begin{proof}[Outline of proof]
We recall the following well-known representations of spheres:
\begin{alignat*}{2}
Sp(n)/Sp(n-1)
& \overset{\sim}{\to}
  U(2n)/U(2n-1) 
&& \simeq
  S^{4n-1},
\\
  U(n)/U(n-1)
& \overset{\sim}{\to}
SO(2n)/SO(2n-1)
&& \simeq
  S^{2n-1},
\\
Spin(5)/Spin(3)
& \overset{\sim}{\to}
  Spin(7)/G_2
&& \simeq
  S^7.  
\end{alignat*}
Then the trick in \cite[Lemma 5.1]{xkdecomp} shows
 that the natural inclusion map
$G'/L' \hookrightarrow G/L$ is in fact surjective for any of
 the specific triples $(G,G',L)$ in Theorem \ref{thm:Aqai},
where we set $L' := G' \cap L$.
Further, 
we have 
$\mathfrak{g}'_{\mathbb{C}}+\mathfrak{q}=\mathfrak{g}_{\mathbb{C}}$
so that 
the inclusion 
$
     {\mathfrak {g}}_{\mathbb{C}}' 
     \hookrightarrow {\mathfrak {g}}_{\mathbb{C}}
$
induces the bijection
$\mathfrak{g}'_{\mathbb{C}}/\mathfrak{q}'\overset{\sim}{\to}
 \mathfrak{g}_{\mathbb{C}}/\mathfrak{q}
$
and $L'$ coincides with
$N_{G'}(\mathfrak{q}')$.
Thus,
 the diffeomorphism $G'/L'\overset{\sim}{\to}G/L$ is biholomorphic.  
In turn,
 we get an isomorphism of canonical line bundles 
(see \eqref{eqn:OmegaGL}):
\[
\begin{matrix}
    G'\times_{L'}\mathbb{C}_{2\rho(\mathfrak{u}')}
   &\overset{\sim}{\to}
   &G\times_{L}\mathbb{C}_{2\rho(\mathfrak{u})}
\\
   \downarrow && \downarrow
\\
    G'/L'
   &\overset{\sim}{\to}
   &G/L
\end{matrix}
\]
This implies
\[
\rho(\mathfrak{u})|_{\mathfrak{l}'} = \rho(\mathfrak{u}')
\]
in the setting of Theorem \ref{thm:Aqai}.
Let ${\mathcal {L}}_{\lambda}$ be a $G$-equivariant 
 holomorphic line bundle over $G/L$
 for $\lambda \in \sqrt{-1} {\mathfrak {l}}^{\ast}$.  
Then the pull-back of 
 $\mathcal{L}_{\lambda + 2\rho({\mathfrak {u}})}$
 to $G'/L'$ yields a $G'$-equivariant holomorphic 
line bundle
$\mathcal{L}_{\lambda|_{\mathfrak {l}'} + 2\rho({\mathfrak {u}'})}$
over $G'/L'$.
Hence,
we have natural isomorphisms
\[
H_{\bar{\partial}}^* (G/L,\mathcal{L}_{\lambda + 2\rho({\mathfrak {u}})})
\overset{\sim}{\to}
H_{\bar{\partial}}^* (G'/L',\mathcal{L}_
{\lambda|_{\mathfrak {l}'} +2\rho({\mathfrak {u}'})}
)
\]
between Dolbeault cohomology groups.  
Thus, we get Theorem \ref{thm:Aqai} in view of the geometric interpretation of
Zuckerman derived functor modules
(see Section \ref{sec:appendix}).
\end{proof}

\begin{remark}\label{rem:Aqai}
1)\enspace
The pairs $(G,G')$  in Theorem
\ref{thm:Aqai} are reductive symmetric pairs except for the case
$(G,G') = (SO_0(4,3),G_2(\mathbb{R}))$.

2)\enspace
The pair $(\mathfrak{g},\mathfrak{l})$ is a reductive
symmetric pair in all the cases of Theorem \ref{thm:Aqai}
(${\mathfrak {q}}$ is of symmetric type 
in the sense of Definition \ref{def:qsymm}). 
Correspondingly there are two choices of $\theta$-stable parabolic
subalgebras $\mathfrak{q}$ of $\mathfrak{g}_{\mathbb{C}}$ with
$N_G(\mathfrak{q})\simeq L$.
In either case,
$\overline{A_{\mathfrak q}(\lambda)}|_{G'}$ is almost irreducible.

3)\enspace
In the compact case (i.e. the last three rows),
 the restriction 
 is irreducible for all $\lambda$.  
\end{remark}

\begin{example}\label{ex:Aq}
1)\enspace
In \cite{xkdecomp} we gave a different proof of
Theorem~\ref{thm:Aqai}
for the pair
$SO_0(4,3)\downarrow G_2(\mathbb{R})$
 based on the Beilinson--Bernstein localization theory,
 and then applied it to 
 construct (all) discrete series representations for
non-symmetric homogeneous spaces
$G_2(\mathbb{R})/SL(3,\mathbb{R})$
and
$G_2(\mathbb{R})/SU(2,1)$.

2)\enspace
H. Sekiguchi applied the restriction of $A_{\mathfrak{q}}(\lambda)$ 
 with respect to the symmetric pair
$U(n,n) \downarrow Sp(n,\mathbb{R})$
for more general ${\mathfrak {q}}$
to get a range characterization theorem of the Penrose transform (see \cite{spen}).
Following the notation in \cite[Proposition 1.5]{spen},
 we see that
the unitary character ${\mathbb{C}}_{\lambda}$ is 
 in the weakly fair range
  for the $\theta$-stable maximal parabolic subalgebra ${\mathfrak {q}}$
 considered in Theorem \ref{thm:Aqai}
if and only if $\lambda=\lambda_1e_1$ with $\lambda_1\ge-n$.
Further, 
 $A_{\mathfrak{q}}(\lambda)$ 
is irreducible as a ${\mathfrak {u}}(n,n)$-module
 for all $\lambda_1 \ge -n$.  
Its restriction to ${\mathfrak {sp}}(n,{\mathbb{R}})$
 stays irreducible for $\lambda_1>-n$,
 but splits into two irreducible modules
$(W(n,1)_+)_K\oplus(W(n,1)_-)_K$.

3)\enspace
Dunne and Zierau \cite{xduzi} determined the automorphism groups
  of elliptic orbits.
It follows from their results that
 our list in Theorem \ref{thm:Aqai} exhausts all
 the cases where $\overline{A_{\mathfrak q}(\lambda)}|_{G'}$
 stays irreducible for sufficiently positive $\lambda$.  
\end{example}

\subsection{Irreducible restriction $\pi|_{G'}$ with
$\pi=\mathrm{Ind}_P^G(\tau)$}
\label{subsec:3.4}

This subsection discusses for which triples $(G,G',\pi)$
 the restriction $\pi|_{G'}$ is 
 (almost) irreducible in the setting 
 that $\pi$ is a (degenerate) principal series
representation $\pi=\operatorname{Ind}_P^G(\tau)$ of $G$.

Let $P$ be a parabolic subgroup of $G$ with Levi decomposition
$P=LN$. 
For an irreducible unitary representation $\tau$ of $L$,
we extend it to $P$ by letting $N$ act trivially,
and denote by $\operatorname{Ind}_P^G(\tau)$ the unitarily induced
representation of $G$.

\begin{theorem}\label{thm:aips}
Suppose that $(G,G',L)$ is one of the following triples:
\[
\begin{array}{ccc}
G & G' & L
\\
\hline
SL(2n,\mathbb{C})  & Sp(n,\mathbb{C})  & GL(2n-1,\mathbb{C})
\\
SO(2n,\mathbb{C})  & SO(2n-1,\mathbb{C})  & GL(n,\mathbb{C})
\\
SO(7,\mathbb{C})  & G_2(\mathbb{C})  & \mathbb{C}^\times \times SO(5,\mathbb{C})
\\
\hline
SL(2n,\mathbb{R})  &Sp(n,\mathbb{R})  & GL(2n-1,\mathbb{R})
\\
SO(n,n)  & SO(n,n-1)  &  GL(n,\mathbb{R})
\\
SO(4,3)  & G_2(\mathbb{R})  & SO(1,1) \times SO(3,2)
\end{array}
\]
Then,
the degenerate unitary principal series representations
$\pi = \operatorname{Ind}_P^G(\tau)$
of $G$ are almost irreducible when restricted to the subgroup $G'$ for
any one dimensional unitary representation $\tau$
 of any parabolic subgroup $P$ 
 having $L$ as its Levi part.
\end{theorem}

\begin{proof}[Outline of the proof]
The subgroup $G'$ acts transitively on the (real) flag variety $G/P$
in the setting of Theorem \ref{thm:aips},
and the isotropy subgroup $P':=G'\cap P$ becomes a parabolic subgroup
of $G'$.
Then we get an isomorphism $G'/P'\overset{\sim}{\to}G/P$,
 and hence the conclusion follows. 
\end{proof}

We note
 that the parabolic subgroup $P$ in Theorem \ref{thm:aips} is maximal.

\begin{example}\label{ex:aips}
For simplicity,
 we use $GL(2n,{\mathbb{R}})$ 
instead of the semisimple group $SL(2n,{\mathbb{R}})$
 in the fourth row,
 and consider the reductive symmetric pair
$(G,G')=(GL(2n,\mathbb{R}),Sp(n,\mathbb{R}))$.  
Let $P$ be a maximal parabolic subgroup $P$ of $G$ with
Levi subgroup $L=GL(2n-1,\mathbb{R}) \times GL(1,{\mathbb{R}})$.  
Then $P$ has an abelian unipotent radical $\mathbb{R}^{2n-1}$
and $P'=G'\cap P$ has a non-abelian unipotent radical which is
isomorphic to the Heisenberg group $H^{2n-1}$.
In this case
the unitary representation 
$\pi =\operatorname{Ind}_P^G(\tau)$
 is irreducible as a representation of $G$
 for any unitary character $\tau$ of $P$.  
On the other hand,
the restriction of $\pi$
 to $G'$ is more delicate.  
It stays irreducible for generic $\tau$ 
(i.e. $d \tau \ne 0$)
and splits into
two irreducible representations of $G'$ for singular $\tau$,
giving rise to a `special unipotent representation'
of $G'=Sp(n,\mathbb{R})$.
See \cite{GASUR} 
 for a detailed analysis in connection with the Weyl operator calculus.  
\end{example}

\begin{remark}\label{rem:aips}
For a complex reductive group,
the underlying $(\mathfrak{g},K)$-modules of (degenerate) principal
series representations are isomorphic to some
$A_{\mathfrak{q}}(\lambda)$. 
Thus the first three cases in Theorem \ref{thm:aips} have already
appeared in Theorem \ref{thm:Aqai} in the context of
$A_{\mathfrak{q}}(\lambda)$. 
\end{remark}

\subsection{Irreducible restriction of minimal
representation}
\label{subsec:nilp}

Thirdly, we present an example of almost irreducible branching laws
for representations $\pi$ which are supposed to be attached to
minimal nilpotent coadjoint orbits.

Let $\varpi$ be the irreducible unitary representation of the
indefinite orthogonal group
$G = O(p,q)$ for $p,q \ge 2$, $(p,q) \ne (2,2)$ and $p+q$ even,
constructed in \cite{xbz} or \cite[Part I]{xkor:2}.
It is a representation of Gelfand--Kirillov dimension $p+q-3$,
and is \emph{minimal} in the sense that its annihilator in the
enveloping algebra $U(\mathfrak{g})$ is the Joseph ideal if $p+q>6$.
\begin{theorem}[$O(p,q) \downarrow O(p,q-1)$]
\label{thm:ain}
\[
\varpi|_{O(p,q-1)} \simeq V_+ + V_-
\]
where $V_\pm$ are irreducible representations of 
$O(p,q-1)$.
\end{theorem}

\begin{proof}
See \cite[Corollary 7.2.1]{xkor:2}.  
\end{proof}

\begin{remark}
\label{rem:ain}
The irreducible decomposition $V_+ + V_-$
 has a geometric meaning
 in connection to the smallest $L^2$-eigenvalues
 of the (ultra-hyperbolic) Laplacian on pseudo-Riemannian space forms.
\end{remark}

\section{Multiplicity-free conjecture}
\label{subsec:mfconj}

Irreducible restrictions to reductive subgroups are a somewhat
rare phenomenon, as we have seen in the previous section.
On the other hand, it happens more often that
the restriction is multiplicity-free
 with respect to reductive symmetric pairs
 $(G, G')$ (see \cite{xrims40}
 for examples).
In this section, we propose a conjectural sufficient condition
 for the restriction $\pi|_{G'}$ to be multiplicity-free
  in the setting where
 $\pi_K$ is a Zuckerman derived functor module $A_{\mathfrak q}(\lambda)$.
Our conjecture is motivated by the propagation theorem
 of multiplicity-free property
 under \lq{visible actions}\rq\
 \cite{xvissym}.  

\begin{definition}\label{def:qsymm}
1)\enspace 
We say a $\theta$-stable parabolic subalgebra
 $\mathfrak{q}={\mathfrak {l}}+ {\mathfrak {u}}$ is of 
\textit{symmetric type}
 if $(\mathfrak{g},\mathfrak{l})$ forms a
symmetric pair.

2)\enspace
We say that $\mathfrak{q}$ is of \textit{virtually symmetric type} if there
exists a $\theta$-stable parabolic subalgebra
$\widetilde{\mathfrak{q}}$ of symmetric type such that
$\widetilde{L}/L\equiv N_G(\widetilde{\mathfrak{q}})/N_G(\mathfrak{q})$
is compact. 
\end{definition}
\begin{remarkwn}
1) \enspace
If $\mathfrak{q}$ is of virtually symmetric type,
then we have a fibration 
$
\widetilde{L}/L \to G/L \to G/\widetilde{L}
$
 with compact fiber $\widetilde L/L$.  

2)\enspace
If $\mathfrak{q}$ is of symmetric type,
then
$\mathfrak{q}$ is obviously of virtually symmetric type.

3)\enspace
Any parabolic subalgebra is of virtually symmetric type
 if $G$ is compact.
\end{remarkwn}

Let $\mathfrak{q}$ be a $\theta$-stable
parabolic subalgebra of $\mathfrak{g}_{\mathbb{C}}$,
and $\overline{A_{\mathfrak {q}}(\lambda)}$
 be the unitarization of
$A_{\mathfrak{q}}(\lambda)$. 
Suppose $(\mathfrak{g},\mathfrak{g}')$ is a reductive symmetric pair.
We then propose the following two conjectures:
\begin{conjecture}\label{conj:mf}
If a $\theta$-stable parabolic subalgebra $\mathfrak{q}$ is of
symmetric type,
then the restriction $\overline{A_{\mathfrak {q}}(\lambda)}|_{G'}$
 is multiplicity-free for
sufficiently regular $\lambda$.
\end{conjecture}
\begin{conjecture}\label{conj:bdd}
If $\mathfrak{q}$ is of virtually symmetric type,
then the restriction
$\overline{A_{\mathfrak {q}}(\lambda)}|_{G'}$
 has a uniformly bounded multiplicity.
\end{conjecture}

Here are some affirmative cases:
\begin{example}
\label{ex:holo}
Suppose $G$ is a non-compact simple Lie group
 such that $G/K$ is a Hermitian symmetric space.  
We write ${\mathfrak{g}}={\mathfrak {k}}+{\mathfrak {p}}$
 for the Cartan decomposition.  
Then ${\mathfrak {p}}_{\mathbb{C}}:={\mathfrak {p}}\otimes_{\mathbb{R}}
{\mathbb{C}}$ decomposes
 into a direct sum
 of two irreducible representations of $K$, 
say ${\mathfrak {p}}_{\mathbb{C}}={\mathfrak {p}}_+ \oplus {\mathfrak {p}}_-$.  
Then ${\mathfrak {q}}:={\mathfrak {k}}_{\mathbb{C}}+{\mathfrak {p}}_+$
 is a $\theta$-stable parabolic subalgebra
 of symmetric type.  
If $\lambda$ is in the good range,
 then $A_{\mathfrak {q}}(\lambda)$ is the underlying 
 $({\mathfrak {g}}, K)$-module
 of a holomorphic discrete series representation
 of scalar type.  
In this case,
 we see Conjecture \ref{conj:mf} holds
 by the explicit branching law:
\begin{alignat*}{2}
&G'=K
 &&\cdots\,\,
 \text{Hua \cite{xhua}, Kostant, Schmid \cite{xschthe}},
\\
&G':\text{ non-compact} \,\,&&\cdots\,\, \text{Kobayashi \cite{xrims40}}.  
\end{alignat*}
\end{example}
\begin{example}\label{ex:mfhw}
As a generalization of Example \ref{ex:holo},
 we retain that $G/K$ is a Hermitian 
 symmetric space,
 and assume that ${\mathfrak {q}}$ is of holomorphic type
 in the sense that ${\mathfrak {q}} \cap {\mathfrak {p}}_{\mathbb{C}}
 \supset {\mathfrak {p}}_+$.  
Then $A_{\mathfrak {q}}(\lambda)$
 is at most a finite
 direct sum of irreducible
 unitary highest weight modules
 if $\lambda$ is in the weakly fair range
 (see \cite{xadams}).  
In this case,
 Conjecture \ref{conj:bdd} is true
 for any $A_{\mathfrak {q}}(\lambda)$
(see \cite[Theorem B]{mf-korea}).
Further, it was proved in \cite[Theorems A, C]{mf-korea}
 as a special case of the propagation theorem
 of multiplicity-free property 
 that the restriction $\pi|_{G'}$
 is multiplicity-free
 if $\pi$ is
 an irreducible unitary highest weight module of scalar type.
\end{example}
\begin{example}\label{ex:bdd}
For
$(G,G')=(O(p,q),O(r)\times O(p-r,q))$
and for a $\theta$-stable parabolic subalgebra $\mathfrak{q}$ of
maximal dimension,
we see from explicit branching laws \cite{xk:1} that Conjecture
\ref{conj:mf} holds in this case.
Likewise, Conjecture \ref{conj:mf} holds for 
 the restriction $O(2p,2q)\downarrow U(p,q)$
 again by explicit branching laws \cite{xkdecomp}.
\end{example}

\begin{example}
\label{ex:cptmf}
For any compact group $G$, 
 the restriction $\pi_{\lambda}|_{G'}$
 is always multiplicity-free
 if ${\mathfrak {q}}$ is of symmetric type 
 (\cite[Theorems E, F]{mf-korea})
 and hence,
 Conjecture \ref{conj:mf} is true.  
\end{example}

\begin{example}
\label{ex:GL}
Let $(G,G')=(GL(2n,{\mathbb{C}}),
 GL(n,{\mathbb{C}}) \times GL(n, {\mathbb{C}}))$, 
 and ${\mathfrak {q}}$ a $\theta$-stable
 parabolic subalgebra
 such that $N_G({\mathfrak {q}}) \simeq G'$.  
Then ${\mathfrak {q}}$ is of symmetric type.  
Further,
 we have the following unitary equivalence:
\[
  \overline{A_{{\mathfrak {q}}}(\lambda)}|_{GL(n, \mathbb C)\times GL(n, \mathbb C)}
 \simeq L^2(GL(n,{\mathbb{C}})).  
\]
Thanks to the Plancherel formula
 of the group $GL(n,{\mathbb{C}})$ due to the 
 Gelfand school
 and Harish-Chandra,
 we see that Conjecture \ref{conj:mf}
 holds also in this case.  
 
 Let us retain the same $\theta$-stable parabolic subalgebra $\mathfrak q$
 and  consider another reductive symmetric pair
 $(G, G'') = (GL(2n, \mathbb C), GL(2n, \mathbb R))$.
Then, we get the following unitary isomorphism:
 \[
  \overline{A_{{\mathfrak {q}}}(\lambda)}|_{GL(2n, \mathbb R)}
 \simeq L^2(GL(2n,{\mathbb{R}})/GL(n, \mathbb C)).  
\]
Again, the right-hand side is multiplicity-free
 by the Plancherel formula for reductive symmetric space
 due to Oshima, van den Ban, Schlichtkrull, and Delorme \cite{xdelorme}
 among others. 
 (It should be noted that the Plancherel formula
for a reductive symmetric space is 
not multiplicity-free
 in general.)
\end{example}

\begin{remark}\label{rem;mf}
As we have seen in Example~\ref{ex:GL},
Conjectures \ref{conj:mf} and \ref{conj:bdd} refer to the
multiplicities in both discrete and continuous spectrum in the
branching law $\overline{A_{\mathfrak {q}}(\lambda)}|_{G'}$.
\end{remark}

\section{Discretely decomposable branching laws}
\label{sec:deco}

This section highlights another nice class of branching
problems, 
namely,
when the restriction $\pi|_{G'}$ splits discretely without continuous
spectrum. 

An obvious case is when $\dim \pi < \infty$
 or when $G'$ is compact.  
One of the advantages of discretely decomposable restrictions 
 is that we can expect a
combinatorial and detailed study of branching laws by purely
algebraic methods 
because we do not have analytic difficulties 
 arising from continuous spectrum.

Prior to \cite{xk:1}, 
discretely
 decomposable restrictions $\pi|_{G'}$ were known
 in some specific settings, 
e.g.\ the $\theta$-correspondence
 for the Weil representation
 with respect to compact dual pair \cite{xhowe}, 
 or when $\pi$ is a holomorphic discrete series representation
 and $G'$ is a Hermitian Lie group 
 \cite{xjv}.
A systematic study in the general case including Zuckerman derived
functor modules $A_{\mathfrak{q}}(\lambda)$
was initiated by the author in a series of papers
\cite{xkdecomp,xkdecoalg,xkdecoass,xkshunki,xkicm}.
See \cite{xgrwaII,xk:1,xkdecomp,xkor:2,xspeh}
for a number of concrete examples of branching laws $\pi|_{G'}$ in 
this framework,
\cite{xoda-so} for some application to modular symbols, 
 \cite{xkdisc} for the construction
 of new discrete series representations
 on non-symmetric spaces.
See also the lecture notes \cite{deco-euro} 
 for a survey on representation theoretic aspects,
and \cite{xkshunki, xkicm} for some applications.

In this section,
 we give a brief overview of discretely decomposable restrictions
 including some recent developments
 and open problems.

\subsection{Infinitesimally discretely decomposable restrictions}
\label{sec:infdeco}

Let us begin with an algebraic formulation.  
Suppose ${\mathfrak {g}}'$ is a Lie algebra.  
\begin{definition}
\label{def:algdeco}
A $\mathfrak{g}'$-module $V$ is said to be 
\textit{discretely decomposable} if there exists an increasing
filtration $\{V_n\}$ such that
$V=\bigcup_{n=0}^\infty V_n$
and each $V_n$ is of finite length as a $\mathfrak{g}'$-module.
\end{definition}

In the setting
 where $G'$ is a real reductive Lie group with maximal compact subgroup
$K'$, 
the terminology \lq{discretely decomposable}\rq\
 fits well if $V$ is a unitarizable $({\mathfrak {g}}', K')$-module,
 namely, 
if $V$ is the underlying $(\mathfrak{g}',K')$-module of a unitary
representation of $G'$:
\begin{remark}
[{\cite[Lemma 1.3]{xkdecoass}}]
\label{rem:algdeco}
Suppose $V$ is a unitarizable $(\mathfrak{g}',K')$-module.
Then $V$ is discretely decomposable
as a ${\mathfrak {g}}'$-module
 if and only if $V$ is decomposed into an 
 algebraic direct sum of irreducible $(\mathfrak{g}',K')$-modules.
\end{remark}

We apply Definition \ref{def:algdeco}
 to branching problems.
Let $G$ be a real reductive Lie group,
and $G'$ a reductive subgroup of $G$.
We may and do assume that $K$ is a maximal compact subgroup of $G$
 and $K':=K \cap G'$
 is that of $G'$.  
\begin{definition}\label{def:decores}
Let $\pi$ be a unitary representation of $G$
of finite length.
We say the restriction $\pi|_{G'}$ is {\it{infinitesimally discretely
decomposable}}
 if the underlying $(\mathfrak{g},K)$-module $\pi_K$ is
discretely decomposable as a $\mathfrak{g}'$-module.
\end{definition}

Here is a comparison between the category
 of unitary representations
 and that of $({\mathfrak {g}},K)$-modules:
\begin{conjecture}
\label{conj:deco}
Let $\pi$ be an irreducible unitary representation of $G$,
and $G'$ a reductive subgroup of $G$.
Then the following two conditions on $(G,G',\pi)$ are equivalent:
\begin{enumerate}[\upshape(i)]
\item  
The restriction $\pi|_{G'}$ is infinitesimally discretely decomposable.
\item  
The unitary representation
 $\pi$ decomposes discretely into a direct sum of irreducible unitary
representations of $G'$.
\end{enumerate}
\end{conjecture}

In general,
 the implication (i) $\Rightarrow$ (ii) holds.
Moreover, the branching law
 for the restriction of the unitary representation $\pi$
 to $G'$
 and that for the restriction of the $({\mathfrak {g}},K)$-module
 $\pi_K$ to $({\mathfrak {g}}',K')$ are 
 essentially the same under the assumption 
 (i) (see \cite[Theorem 2.7]{xkdecoaspm}).
The converse statement (ii) $\Rightarrow$ (i)
 remains open;
 affirmative results have
 been partially obtained by Duflo and Vargas \cite{xdv}
 for discrete series representations $\pi$,
 see also \cite[Conjecture D]{xkdecoaspm}
 and \cite{xzhuliang}.

For the study of discretely decomposable
 restrictions,
 the concept of $K'$-admissible restrictions is useful:
\begin{proposition}
\label{prop:admKprime}
If the restriction $\pi|_{K'}$ is $K'$-admissible
 then both the conditions (i) and (ii) in Conjecture \ref{conj:deco}
 hold.  
\end{proposition}

\begin{proof}
See \cite[Proposition 1.6]{xkdecoass}
 and \cite[Theorem 1.2]{xkdecomp},
 respectively.  
\end{proof}
\subsection{Analytic approach}
\label{sec:ASK}

We now consider a criterion for the $K'$-admissibility
 of a representation $\pi$. 

Let $K'$ be a closed subgroup of $K$.
Associated to the Hamiltonian $K$-action on the cotangent bundle
$T^*(K/K')$,
we consider the momentum map
\[
\mu: T^*(K/K') \to \sqrt{-1} \mathfrak{k}^*.
\]
Then its image equals 
$\sqrt{-1} \operatorname{Ad}^*(K)({\mathfrak {k}}')^{\perp}$, 
 where
$(\mathfrak{k}')^\bot$ is the kernel of the projection
$
\operatorname{pr}_{\mathfrak{k}\to\mathfrak{k}'}:
\mathfrak{k}^*\to(\mathfrak{k}')^*
$,
the dual to the inclusion $\mathfrak{k}' \subset \mathfrak{k}$ of Lie
algebras. 
The momentum set $C_K(K')$ is defined
 as the intersection of $\operatorname{Image}\mu$
 with a dominant Weyl chamber $C_+$
 $(\subset \sqrt{-1}{\mathfrak {t}}^*)$
 with respect to a fixed positive system
 $\Delta^+({\mathfrak {k}}, {\mathfrak {t}})$
 and a Cartan subalgebra ${\mathfrak {t}}$ of ${\mathfrak {k}}$:
\begin{equation}\label{eqn:6.3.1}
C_K(K')
:=\Cp\cap\sqrt{-1}\operatorname{Ad}^*(K)(\mathfrak{k}')^\bot.
\end{equation}
Here we regard ${\mathfrak {t}}^*$
 as a subspace of ${\mathfrak {k}}^*$
 via a $K$-invariant non-degenerate
 bilinear form on ${\mathfrak {k}}$.  

Next,
 let $\pi$ be a $K$-module.  
We write $\operatorname{AS}_K(\pi)$ for the asymptotic $K$-support
 introduced by Kashiwara and Vergne
\cite{xkashiv}, 
 that is, 
the limit cone of the set of highest weights
 of $K$-types in $\pi$.  
$\operatorname{AS}_K(\pi)$ is a closed cone in $C_+$.  

We are ready to state 
 a criterion for admissible restrictions.

\begin{theorem}
\label{thm:sufadm}
Let $G\supset G'$ be 
a pair of reductive Lie groups, 
 and take maximal compact subgroups 
$K\supset K'$, respectively. 
Suppose $\pi$ is an irreducible unitary representation of $G$.  

\begin{itemize}
\item[\upshape1)]
Then the following two conditions are equivalent:
\begin{itemize}
\item[\upshape(i)]
$C_K(K')\cap \operatorname{AS}_K(\pi) = \{ 0\}$.
\item[\upshape (ii)] The restriction $\pi|_{K'}$ is $K'$-admissible.
\end{itemize}
\item[\upshape2)]
If one of the equivalent conditions (i) or (ii)
 is fulfilled,
 then the restriction $\pi|_{G'}$ is infinitesimally discretely
 decomposable
 (see Definition \ref{def:decores}), 
 and the restriction $\pi|_{G'}$ is unitarily equivalent
 to the Hilbert direct sum:
\[
\pi|_{G'}\simeq
\sideset{}{^{\oplus}}\sum_{\tau\in G'} n_\pi(\tau)\tau
\qquad
\text{with }
n_\pi(\tau)<\infty
\quad
\text{for any}
\quad
\tau\in\widehat G'.    
\]
\end{itemize}
\end{theorem}

\begin{proof}
[Outline of Proof]
The proof of the implication (i) $\Rightarrow$ (ii)
 was proved first by the author \cite[Theorem 2.8]{xkdecoalg}
 by using the singularity spectrum of hyperfunction
 characters
 in a more general setting 
 where $\pi$ is just a $K$-module
 such that the multiplicity
\[
  m_{\pi}(\tau):=\dim\operatorname{Hom}_K(\tau, \pi)
\]
is of infra-exponential growth.  
In the same spirit, 
Hansen, Hilgert, and Keliny \cite{xhaetal}
gave an alternative proof 
 by using the wave front set of distribution characters
under the assumption
 that $m_{\pi}(\tau)$ is at most of polynomial growth.  
The last statement was proved
 in \cite[Theorem 2.9]{xkdecoalg} as a consequence
 of Proposition \ref{prop:admKprime}.  
See also \cite{deco-euro}.  
\end{proof}

The condition (i) in Theorem \ref{thm:sufadm}
 is obviously fulfilled if 
$C_K(K')=\{ 0\}$ or if $\operatorname{AS}_K(\pi)=\{ 0 \}$. 
We pin down the meanings of these extremal cases: 

\medskip
\noindent
1) $C_K(K')=\{ 0\} \Leftrightarrow K'=K$.
Then the conclusion in Theorem \ref{thm:sufadm} 2) is 
nothing but Harish-Chandra's admissibility theorem
(see Fact~\ref{fact:HCadm}).

\medskip
\noindent
2) $\operatorname{AS}_K(\pi)=\{ 0\} \Leftrightarrow \dim \pi<\infty$.

\subsection{Algebraic approach}
\label{sec:AV}
For a finitely generated $\mathfrak{g}$-module $X$,
the associated variety 
$\mathcal{V}_{\mathfrak{g}_{\mathbb{C}}}(X)$
 is a subvariety in the nilpotent cone
$\mathcal{N}_{\mathfrak{g}_{\mathbb{C}}}$
of $\mathfrak{g}^*_{\mathbb{C}}$
(see \cite{xvoganass}).
In what follows,
 let $X$ be the underlying $(\mathfrak{g},K)$-module
 of $\pi \in \widehat G$
 and $Y$ the underlying $(\mathfrak{g}',K')$-module
 of $\tau(\in \widehat {G'})$.

We write 
$
\operatorname{pr}_{\mathfrak{g}\to \mathfrak{g}'} :
  \mathfrak{g}_{\mathbb{C}}^* \to 
 (\mathfrak{g}_{\mathbb{C}}^\prime)^*
$
for the natural projection dual to
$ \mathfrak{g}_{\mathbb{C}}^\prime \hookrightarrow
  \mathfrak{g}_{\mathbb{C}}$.
\begin{theorem}[see {\cite[Theorem 3.1]{xkdecoass}}]
 \label{thm:5.2.1}
If \/
$\operatorname{Hom}_{\mathfrak{g}'}(Y,X) \ne \{0\}$, 
then 
  \begin{equation}
    \label{eqn:5.2.1}
    \operatorname{pr}_{\mathfrak{g} \to \mathfrak{g}'}
     (\mathcal{V}_{\mathfrak{g}_{\mathbb{C}}}(X)) \subset
\mathcal{V}_{\mathfrak{g}^\prime_{\mathbb{C}}}(Y).
  \end{equation}
\end{theorem}
Theorem~\ref{thm:5.2.1}
leads us to
 a useful criterion for discrete
decomposability by means of associated varieties:
\begin{corollary} \label{cor:necdeco} 
If the restriction $X$
is infinitesimally discretely decomposable
 as a ${\mathfrak {g}}'$-module,
 then
$
  \operatorname{pr}_{\mathfrak{g} \to \mathfrak{g}'}
   (\mathcal{V}_{\mathfrak{g}_{\mathbb{C}}}(X))
$
is contained in the nilpotent cone of
$\mathfrak{g}_{\mathbb{C}}^\prime$.
\end{corollary}

\begin{remark}\label{rem:AV}
An analogous statement to Theorem \ref{thm:5.2.1} fails if we replace
$\operatorname{Hom}_{\mathfrak{g}'}(Y,X) \ne \{0\}$
by
$\operatorname{Hom}_{G'}(\tau,\pi|_{G'}) \ne \{0\}$.  
\end{remark}

\begin{remark}
\label{rem:AV2}
Analogous results
 to Theorem \ref{thm:5.2.1}
 and Corollary \ref{cor:necdeco} hold
 in the category ${\mathcal {O}}$. 
See \cite{xverma}.  
\end{remark}

It is plausible 
 that the following holds:
\begin{conjecture}\label{conj:AV}
The inclusion \eqref{eqn:5.2.1} in Theorem \ref{thm:5.2.1}
 is equality.
\end{conjecture}

Here are some affirmative results to Conjecture \ref{conj:AV}.
\begin{proposition}\label{prop:AV}
\ 
\begin{itemize}
\item[\upshape{1)}]
$X$ is the Segal--Shale--Weil representation,
and 
$\mathfrak{g}' = \mathfrak{g}'_1 \oplus \mathfrak{g}'_2$
is the compact dual pair in\/
$\mathfrak{g} = \mathfrak{sp}(n,\mathbb{R})$.
\item[\upshape{2)}]
$X$ is the underlying $(\mathfrak{g},K)$-module of the minimal
representation of $O(p,q)$ ($p+q$ even),
and $(\mathfrak{g},\mathfrak{g}')$ is a symmetric pair.
\item[\upshape{3)}]
$X$ is a (generalized) Verma module,
and $(\mathfrak{g},\mathfrak{g}')$ is a symmetric pair.
\item[\upshape{4)}]
$X = A_{\mathfrak{q}}(\lambda)$ and $(\mathfrak{g},\mathfrak{g}')$
 is a symmetric pair. 
\end{itemize}
\end{proposition}
\begin{proof}
The first statement could be read off
 from the results in \cite{E-W,NOT} by case-by-case argument
 though they were not formulated by means of Theorem \ref{thm:5.2.1}.
See \cite{xkor:2} for the proof of the second,
 and \cite{xverma} for that of the third statement, respectively.  
The fourth statement is proved recently by Y.~Oshima
 by using a ${\mathcal {D}}$-module argument.  
\end{proof}

\subsection
{Restriction of $A_{\mathfrak {q}}(\lambda)$ to symmetric pair}
\label{sec:symmCone}

For the restriction of $A_{\mathfrak q}(\lambda)$
 to a reductive symmetric pair, 
our criterion is computable.
Let us have a closer look.

Suppose that $(G,G')$ is a symmetric pair defined by an involutive
automorphism $\sigma$ of $G$.
As usual, the differential of $\sigma$ will be denoted
by the same letter.
By taking a conjugation by $G$ if necessary, 
we may and do assume
 that $\sigma$ stabilizes $K$
 and that $\mathfrak{t}$
and $\Delta^+(\mathfrak{k},\mathfrak{t})$ are chosen so that 
\begin{itemize}
\item[1)] 
$\mathfrak{t}^{-\sigma}:= \mathfrak{t}\cap \mathfrak{k}^{-\sigma}$ 
is a maximal 
abelian subspace of $\mathfrak{k}^{-\sigma}$,

\item[2)] 
$\sum^+(\mathfrak{k},\mathfrak{t}^{-\sigma})
:=\{ \lambda|_{\mathfrak{t}^{-\sigma}}:
\lambda\in\Delta^+(\mathfrak{k},\mathfrak{t})\} \setminus \{0\}$
is a positive system of the restricted root system 
$\sum(\mathfrak{k},\mathfrak{t}^{-\sigma})$.
\end{itemize}
Then the momentum set $C_K(K')$ coincides with the dominant Weyl chamber 
 $(\subset\sqrt{-1}(\mathfrak{t}^{-\sigma})^*)$ 
 with respect to 
$\Sigma^+({\mathfrak {k}}, {\mathfrak {t}}^{-\sigma})$.  

Let
$\Delta(\mathfrak{u}\cap\mathfrak{p})\subset\sqrt{-1}\mathfrak{t}^*$ 
be the set of weights in $\mathfrak{u}\cap\mathfrak{p}$, 
 and ${\mathbb{R}}_+\Delta({\mathfrak {u}}\cap {\mathfrak {p}})$
the closed cone spanned by $\Delta({\mathfrak {u}} \cap {\mathfrak {p}})$.  
Then the asymptotic support $\operatorname{AS}_K(A_{\mathfrak q}(\lambda))$
 is contained in ${\mathbb{R}}_+\Delta({\mathfrak {u}}\cap {\mathfrak {p}})$.
\begin{theorem}\label{thm:decosymm}
The following six conditions
 on $({\mathfrak {g}},{\mathfrak {g}}^{\sigma},\mathfrak{q})$
 are equivalent: 
\begin{itemize}
\item[\upshape(i)]
$A_{\mathfrak{q}}(\lambda)$ is non-zero and discretely decomposable
as a $\mathfrak{g}'$-module
for some $\lambda$ in the weakly fair range.
\item[\upshape(i)$'$]
$A_{\mathfrak{q}}(\lambda)$ is discretely decomposable 
as a $\mathfrak{g}'$-module
for any
$\lambda$ in the weakly fair range.
\item[\upshape(ii)]
$\mathbb{R}_+
 \Delta(\mathfrak{u}\cap\mathfrak{p})\cap\sqrt{-1}\mathfrak{t}^{-\sigma}
 = \{0\}$.
\item[\upshape(iii)]
$A_{\mathfrak{q}}(\lambda)$ is non-zero and $K'$-admissible for some
$\lambda$ in the weakly fair range.
\item[\upshape(iii)$'$]
$A_{\mathfrak{q}}(\lambda)$ is $K'$-admissible for any
$\lambda$ in the weakly fair range.
\item[\upshape(iv)]
$\operatorname{pr}_{{\mathfrak {g}} \to {\mathfrak {g}}'}
({\mathcal{V}}_{\mathfrak {g}_{\mathbb{C}}}
  (A_{\mathfrak {q}}(\lambda)))$
 is contained
 in the nilpotent cone of ${\mathfrak {g}}_{\mathbb{C}}'$.  
\end{itemize}
\end{theorem}
\begin{proof}
The equivalences (i) $\Leftrightarrow$ (i)$'$
 and (iii) $\Leftrightarrow$ (iii)$'$
 are easy.  
The implication (ii) $\Rightarrow$ (iii) 
 was first proved in \cite{xkdecomp}.  
Alternatively, we can use Theorem \ref{thm:sufadm}
 and the inclusive relation
 $\operatorname{AS}_K(A_{\mathfrak q}(\lambda))
\subset {\mathbb{R}}_+\Delta({\mathfrak {u}}\cap {\mathfrak {p}})$.
This was the approach taken in \cite{xkdecoalg}.  
Other implications are proved in \cite{xkrons}
 based on Theorem \ref{thm:5.2.1}.  
\end{proof}
See \cite{classAq} for the list of all such triples 
$({\mathfrak {g}}, {\mathfrak {g}}^{\sigma}, {\mathfrak {q}})$.  

\begin{remark}
\label{rem:decosymm}
The implication (i) $\Rightarrow$ (iii)$'$
 and Theorem \ref{thm:sufadm} show that 
\[
  \dim \operatorname{Hom}_{G'} (\tau, \pi|_{G'}) <\infty
 \quad
 \text{for any } 
 \tau \in \widehat {G'}
\]
 if the restriction $\pi|_{G'}$ is infinitesimally discretely 
 decomposable for any $\pi_K \simeq A_{\mathfrak {q}}(\lambda)$.  
\end{remark}

\section{Appendix -- basic properties of $A_{\mathfrak {q}}(\lambda)$}
\label{sec:appendix}

This section gives a quick summary of basic properties on
 Zuckeman's derived functor modules
 and the ``geometric quantization'' of elliptic coadjoint orbits
$\mathcal{O}_\lambda$ in the following scheme:
\[
\begin{array}{cl}
\lambda\in\sqrt{-1}\mathfrak{g}^*&\text{an elliptic and integral element}
\\
\yato&\\
\mathcal{L}_{\lambda+\rho_{\lambda}}\to \mathcal{O}_\lambda
&\text{a $G$-equivariant holomorphic line bundle}\\
\yato&\\
H^*_{\bar \partial}(\mathcal{O}_\lambda,\mathcal{L}_{\lambda+\rho_{\lambda}})
&\text{a Fr\'{e}chet representation of $G$}\\
\yato&\\
\pi_\lambda
&\text{a unitary representation of $G$}
\end{array}
\]

There is no new result in this section,
 and the normalization of the parameters
 and formulation 
 follows the expository notes
 \cite{xkrons, deco-euro}.  
See \cite{xknappv} for a more complete treatment and references therein.

\subsection{Zuckerman derived functor modules}
\label{subsec:ZDF} 
Let $G$ be a connected real reductive Lie group, 
 ${\mathfrak {g}}={\mathfrak {k}}+{\mathfrak {p}}$
 a Cartan decomposition
 of the Lie algebra of ${\mathfrak {g}}$, 
and $\theta$ the corresponding Cartan involution.  

Let $\mathfrak{q}$ be a $\theta$-stable parabolic subalgebra of
$\mathfrak{g}_{\mathbb{C}}$. 
Then the normalizer $L=N_G(\mathfrak{q})$ is a connected reductive
subgroup of $G$,
and the homogeneous space $G/L$ carries a $G$-invariant complex
structure such that the holomorphic tangent bundle
$T(G/L)$ is given as a homogeneous bundle 
$G\times_L(\mathfrak{g}_{\mathbb{C}}/\mathfrak{q})$.
Let $\mathfrak {l}_{\mathbb{C}}$ be the complexification
 of the Lie algebra ${\mathfrak {l}}$ of $L$, 
 and ${\mathfrak {u}}$ the unipotent radical of ${\mathfrak {q}}$.  
Then we have a Levi decomposition 
$\mathfrak{q}=\mathfrak{l}_{\mathbb{C}}+\mathfrak{u}$.  
We set 
$
\rho(\mathfrak{u})(X):= \frac{1}{2} \operatorname{Trace}
   (\operatorname{ad}(X): \mathfrak{u} \to \mathfrak{u})
$
 for $X \in {\mathfrak {l}}$.

We say a Lie algebra homomorphism 
 $\lambda: \mathfrak{l}\to \mathbb{C}$ is \textit{integral} if
$\lambda$ lifts to a character of the connected group $L$,
denoted by $\mathbb{C}_\lambda$.
Then 
$\mathcal{L}_\lambda:=G\times_{G_\lambda}\mathbb{C}_\lambda$
is a $G$-equivariant holomorphic line bundle over $G/L$.
For example, 
 $2 \rho({\mathfrak {u}})$ is integral, 
 and the canonical bundle
$\Omega(G/L) := \Lambda^{top}(T^*(G/L))$
is isomorphic to 
\begin{equation}\label{eqn:OmegaGL}
\Omega(G/L) \simeq {\mathcal {L}}_{2\rho(\mathfrak{u})}
\end{equation}
as a $G$-equivariant holomorphic line bundle.
The Zuckerman derived functor 
$W 
\mapsto
{\mathcal {R}}_{\mathfrak{q}}^j
(W \otimes {\mathbb{C}}_{\rho({\mathfrak {u}})})
$
is a covariant functor 
from the category of $({\mathfrak {l}}, L \cap K)$-modules
 to the category of $({\mathfrak {g}}, K)$-modules.  
We note
 that $L$ is not necessarily compact.  
In this generality,
 H. Wong proved in \cite{xwong}
that the Dolbeault cohomology groups
\[
H_{\bar{\partial}}^j(G/L,{\mathcal{L}}_{\lambda}\otimes \Omega(G/L))
\simeq 
H_{\bar{\partial}}^j(G/L,{\mathcal{L}}_{\lambda+2\rho(\mathfrak{u})})
\]
carry a Fr\'echet topology
 on which $G$ acts continuously
 and that 
$
  {\mathcal{R}}_{\mathfrak {q}}^j({\mathbb{C}}_{\lambda+\rho({\mathfrak {u}})})
$
 are isomorphic to their underlying $(\mathfrak{g},K)$-modules.  
We set
$S:=\dim_{\mathbb{C}}(\mathfrak{u}\cap\mathfrak{k}_{\mathbb{C}})$,
and
\[
A_{\mathfrak{q}}(\lambda) :=
\mathcal{R}_{\mathfrak{q}}^S(\mathbb{C}_{\lambda+\rho(\mathfrak{u})}).
\]
In our normalization,
 $A_{\mathfrak {q}}(0)$ is an irreducible
 and unitarizable $({\mathfrak {g}},K)$-module
 with non-zero $({\mathfrak {g}},K)$-cohomology \cite{xvoza},
 and in particular,
 has the same infinitesimal character with that of the trivial one
 dimensional representation ${\mathbb{C}}$ of $G$.  

\subsection{Geometric quantization of elliptic coadjoint orbit}
\label{sub:7.1.3}

Let $\lambda\in \sqrt{-1}\mathfrak{g}^*$.  
We say that the coadjoint orbit
$
\mathcal{O}_\lambda:=\operatorname{Ad}^*(G)\cdot\lambda
$
is {\it{elliptic}} 
if $\lambda|_{\mathfrak {p}}\equiv 0$.  
We identify ${\mathfrak {g}}$ 
 with the dual space ${\mathfrak {g}}^*$
 by a non-degenerate $G$-invariant bilinear form, 
 and write $X_{\lambda} \in \sqrt{-1} {\mathfrak {g}}$
 for the corresponding element to $\lambda$.
Then $\operatorname{ad}(X_{\lambda})$
 is semisimple
 and all the eigenvalues are pure imaginary.  
The sum of the eigenspaces
 for non-negative eigenvalues
 of $-\sqrt{-1} \operatorname{ad}(X_{\lambda})$
 defines a $\theta$-stable parabolic subalgebra 
 ${\mathfrak {q}}={\mathfrak {l}}_{\mathbb{C}}+ {\mathfrak {u}}$,
 and consequently, 
 the elliptic orbit $\mathcal{O}_\lambda$
carries a $G$-invariant complex structure such that the holomorphic
tangent bundle is given
 by $G \times_L ({\mathfrak {g}}_{\mathbb{C}}/{\mathfrak {q}})$.  

We set $\rho_{\lambda}:= \rho({\mathfrak {u}})$.
If $\lambda + \rho_{\lambda}$ is integral, 
 namely,
 if $\lambda + \rho_{\lambda}$ lifts to a character of $L$, 
 then we can define a $G$-equivariant holomorphic line bundle
$
\mathcal{L}_{\lambda+ \rho_{\lambda}}:=G\times_{L}
\mathbb{C}_{\lambda+\rho_\lambda}
$
over 
$\mathcal{O}_\lambda$.  

Here is a brief summary
 of the important achievements
 on unitary representation theory in 1980s
 and 1990s
 on the geometric quantization
 of elliptic orbits 
 due to Parthasarathy, Zuckerman, Vogan and Wallach
 (algebraic construction, unitarizability
 of Zuckerman derived
functor modules $A_{\mathfrak{q}}(\lambda)$), 
 and Schmid and Wong
 (realization in Dolbeault cohomology, 
 in particular, 
 the closed range property
 of the $\bar \partial$-operator) among others.
See \cite{xknappv, deco-euro} for the original references therein.  
\begin{fact}\label{fact:ZDF}
Let $\lambda\in\sqrt{-1}\mathfrak{g}^*$ be elliptic 
such that $\lambda+ \rho_{\lambda}$ is integral.
\begin{itemize}
\item[\upshape 1)] (vanishing theorem) 
$H^j_{\bar \partial}(\mathcal{O}_\lambda,\mathcal{L}_{\lambda+\rho_{\lambda}})=0$
if $j\not= S$.

\item[\upshape 2)] 
The Dolbeault cohomology group 
$H^S_{\bar \partial}(\mathcal{O}_\lambda,\mathcal{L}_{\lambda+\rho_{\lambda}})$
carries a Fr\'echet topology, on which $G$ acts
continuously.
It is the maximal globalization of
${\mathcal{R}}_{\mathfrak{q}}^S({\mathbb{C}}_\lambda)
=A_{\mathfrak {q}}(\lambda-\rho_{\lambda})$ 
in the sense of Schmid.

\item[\upshape 3)] (unitarizability) 
There is a dense subspace 
$\mathcal{H}$ in $H^S_{\bar \partial}
(\mathcal{O}_\lambda,\mathcal{L}_{\lambda+\rho_{\lambda}})$ on which a 
$G$-invariant Hilbert structure exists.
We denote by $\pi_\lambda$ the resulting unitary representation
on ${\mathcal {H}}$.

\item[\upshape 4)] 
If $\lambda$ is in the good range
 in the sense of Vogan,
 then the unitary representation of $G$ on $\mathcal{H}$
is irreducible and non-zero.
\end{itemize}
\end{fact}

Here,
 by \lq{good range}\rq, 
we mean that $\lambda$ satisfies
\begin{equation}
\label{eqn:good}
  \langle \lambda+ \rho_{\mathfrak {l}}, \alpha\rangle>0
  \quad
  \text{for any }\quad
  \alpha \in \Delta({\mathfrak {u}}, {\mathfrak {h}}_{\mathbb{C}}), 
\end{equation}
where ${\mathfrak {h}}$ is a fundamental Cartan subalgebra containing 
$X_{\lambda}$ and $\rho_{\mathfrak {l}}$
 is half the sum of positive roots
 for 
$
   \Delta({\mathfrak {l}}_{\mathbb{C}},{\mathfrak {h}}_{\mathbb{C}}).  
$
(This condition is independent of the choice 
 of ${\mathfrak {h}}$ and 
 $
  \Delta^+({\mathfrak {l}}_{\mathbb{C}},{\mathfrak {h}}_{\mathbb{C}}).  
$)

\medskip\medskip

\noindent
Graduate School of Mathematical Sciences, IPMU,\\
the University of Tokyo, Komaba, Meguro, Tokyo, 153-8914 Japan\\
toshi@ms.u-tokyo.ac.jp


\begin{thebibliography}{99}
\bibitem{xadams}
J. Adams, 
Unitary highest weight modules, 
Adv. in Math. {\textbf{63}} (1987), 113--137.
%
\bibitem{xban}
 E. van den Ban, 
 Invariant differential operators
 on a semisimple symmetric space
 and finite multiplicities in a Plancherel formula, 
 Arkiv Mat. 
 \textbf{25}
 (1987), 175--187.
%
\bibitem{xbz}
B. Binegar and R. Zierau, 
Unitarization of a singular representation of $SO(p,q)$,
Commun. Math. Phys. {\textbf{138}} (1991), 245--258.
%
\bibitem{xdelorme}
P. Delorme,
Formule de Plancherel pour les espaces sym\'etriques r\'eductifs,
Ann. of Math. (2), {\textbf{147}} (1998),  417--452.

%
\bibitem{xdv}
M. Duflo and J. A. Vargas, 
Branching laws for square integrable representations,
Proc. Japan Acad. Ser. A,
Math. Sci.
\textbf{86}
(2010),
49--54.  
%
\bibitem{xduzi}
E. Dunne and R. Zierau,
The automorphism groups of complex homogeneous spaces.
Math. Ann. {\textbf{307}} (1997), 489--503.
%
\bibitem{E-W}
 T. Enright and J. Willenbring,
Hilbert series, Howe duality and branching for classical groups,
Ann. of Math. (2) \textbf{159} (2004), 337--375.
%
\bibitem{xgg}
I. M. Gelfand and M. I. Graev,
 Geometry of homogeneous spaces, representations of groups in
              homogeneous spaces and related questions of integral geometry. I,
Trudy Moskov. Mat. Ob\v s\v c. 8 (1959), 321--390.

%
\bibitem{xgrwaII}
B. Gross and N. Wallach,
        Restriction of small discrete series representations
        to symmetric subgroups,
        \textit{Proc. Sympos. Pure Math.},
       \textbf{68} (2000),  Amer. Math. Soc., 255--272.
%
\bibitem{xhaetal}
S. Hansen, J. Hilgert and S. Keliny,
Asymptotic $K$-support and restrictions of representations,
Represent. Theory
\textbf{13} (2009),
460--469.  
%
\bibitem{xhowe}
R. Howe,
$\theta$-series and invariant theory,
Proc. Symp. Pure Math.
\textbf{33}
(1979),
Amer. Math. Soc.,
275--285.
%
\bibitem{xhower}
R. Howe, 
Reciprocity laws in the theory of dual pairs,
Progr. in ath. Birkh\"auser, 40 (1983), 159--175
%
\bibitem{xhua}
L. K. Hua,
Harmonic Analysis of Functions of Several Complex
Variables in the Classical Domains,
Amer. Math. Soc.,
1963.  
%
\bibitem{xjv}
H. P. Jakobsen and M. Vergne,
Restrictions and expansions of holomorphic representations,
J. Funct. Anal.
\textbf{34}
(1979),
29--53.  
%
\bibitem{xkashiv}
M. Kashiwara and M. Vergne,
$K$-types and singular spectrum,
In: Lect. Notes in Math.
\textbf{728}, 
1979, 
Springer-Verlag,
177--200.  
%
\bibitem{xknappv}
A. W. Knapp and D. Vogan, Jr., 
Cohomological Induction and Unitary Representations,
Princeton U.P., 
1995.  
%
\bibitem{xkmsj}
T. Kobayashi,
Unitary representations realized in $L^2$-sections
of vector bundles over semisimple symmetric spaces,
Proceedings at the 27-28th Symp. of Functional Analysis
and Real Analysis (1989), Math. Soc. Japan, 
\href{http://www.ms.u-tokyo.ac.jp/~toshi/pub/11.html}{39--54}.
%
\bibitem{xk:1}
T. Kobayashi, 
           The restriction of $A_{\mathfrak{q}}(\lambda)$
            to reductive subgroups,
           \textit{Proc. Japan Acad.},
           {\textbf{69}}
           (1993),
\href{http://projecteuclid.org/euclid.pja/1195511349}{262--267}.  
%
\bibitem{xkvoc}
T. Kobayashi,
Irreducible restriction of $A_{\mathfrak{q}}(\lambda)$ to reductive
subgroups,
Lecture at Summer workshop on representation theory,
Polytechnic University, August 24, 1993.
%
\bibitem{xkdecomp}
T. Kobayashi,
            Discrete decomposability of the restriction
            of $A_{\mathfrak{q}}(\lambda)$
            with respect to reductive subgroups and its application, 
            \textit{Invent. Math.}, \textbf{117} (1994), 
\href{http://dx.doi.org/10.1007/BF01232239}{181--205}.
%
\bibitem{xkdecoalg} 
T. Kobayashi,
Discrete decomposability of the restriction of $A_{\mathfrak{q}}(\lambda)$, 
II. ---micro-local analysis and asymptotic $K$-support, 
\textit{Ann. of Math.},  \textbf{147} (1998), 
\href{http://dx.doi.org/10.2307/120963}{709--729}.
%
\bibitem{xkdecoass} 
T. Kobayashi,
Discrete decomposability of the restriction of $A_{\mathfrak{q}}(\lambda)$, 
III. ---restriction of Harish-Chandra modules and associated varieties,
\textit{Invent. Math.},  \textbf{131} (1998), 
\href{http://dx.doi.org/10.1007/s002220050203}{229--256}.

\bibitem{xkrons}
     T. Kobayashi,
         Harmonic analysis on homogeneous manifolds of reductive type
         and unitary representation theory,
         \textit{Sugaku}, 
         \textbf{46} (1994), 
         Math. Soc. Japan 
         (in Japanese),
          124--143;
         \textit{Translations, Series II}, 
         Selected Papers on Harmonic Analysis, Groups, and Invariants
        (K. Nomizu, ed.),
        \textbf{183} (1998),         Amer. Math. Soc.,
\href{http://www.ms.u-tokyo.ac.jp/~toshi/pub/43.html}{1--31}.

\bibitem{xkdisc}
   T. Kobayashi, 
    Discrete series representations for the orbit spaces
     arising from two involutions of real reductive Lie groups,
     \textit{J. Funct. Anal.},
     \textbf{152} (1998), 
\href{http://dx.doi.org/10.1006/jfan.1997.3128}{100--135}.

\bibitem{xkshunki}
 T. Kobayashi, 
            Theory of discrete decomposable branching laws
            of unitary representations of semisimple Lie groups
            and some applications, 
            \textit{Sugaku}, \textbf{51} (1999), 
Math. Soc. Japan (in Japanese), 
            337--356; English translation, 
            \textit{Sugaku Exposition}, \textbf{18} (2005), 
            Amer. Math. Soc. 
\href{http://www.ms.u-tokyo.ac.jp/~toshi/pub/84.html}{1--37}.

\bibitem{xkdecoaspm} 
 T. Kobayashi,
            Discretely decomposable restrictions
            of unitary representations of reductive Lie groups
            --- examples and conjectures, 
            \textit{Advanced Study in Pure Math.}, 
            \textbf{26}
            (2000), 
\href{http://www.ms.u-tokyo.ac.jp/~toshi/pub/57.html}{98--126}. 
%
\bibitem{xkicm}
 T. Kobayashi,
            Unitary representations and  branching laws,
            \textit{Proceedings of the I.C.M. 2002
            at Beijing}, \textbf{2} (2002), 
\href{http://arxiv.org/abs/math.RT/0304326}{615--627}.
%
\bibitem{deco-euro}
T. Kobayashi, 
Restrictions of unitary representations of real reductive groups,
\textit{Progr. in Math.} \textbf{229}, pages 
\href{http://dx.doi.org/10.1007/b139076}{139--207},
Birkh\"auser, 2005.
%
\bibitem{xrims40}
T. Kobayashi, 
Multiplicity-free representations and visible actions
on complex manifolds, 
Publ. Res. Inst. Math. Sci. 
{\textbf{41}}(2005),
\href{http://dx.doi.org/10.2977/prims/1145475221}{497--549} 
(a special issue of Publications
 of the Research Institute for Mathematical Sciences
 commemorating the fortieth anniversary of
 the founding of the Research Institute for Mathematical Sciences).  
%
\bibitem{mf-korea}
   T. Kobayashi, 
Multiplicity-free theorems of the restrictions of unitary 
highest weight modules with respect to reductive symmetric pairs. 
Progr. Math.
\textbf{255}, 
pages 
\href{http://dx.doi.org/10.1007/978-0-8176-4646-2_3}{45--109}. 
Birkh\"auser, 2007. 
%
\bibitem{xvissym}
T. Kobayashi, 
Visible actions on symmetric spaces. 
Transformation Groups, \textbf{12} (2007), 
\href{http://dx.doi.org/10.1007/s00031-007-0057-4}{671--694}.
%
\bibitem{xverma}
T. Kobayashi, Restrictions of generalized Verma modules to symmetric pairs, 
submitted,
\href{http://arxiv.org/abs/1008.4544}{arXiv:1008.4544}
%
\bibitem{xoda-so}
   T. Kobayashi and T. Oda,
    Vanishing theorem of modular symbols on locally symmetric spaces,
    \textit{Comment. Math. Helvetici},
   \textbf{73} (1998), 
\href{http://dx.doi.org/10.1007/s000140050045}{45--70}.
%
\bibitem{xkor:2} 
T. Kobayashi and B. \O rsted, 
Analysis on minimal representations of $O(p,q)$, 
Part II. Branching Laws, 
\textit{Adv. in Math.}, \textbf{180} (2003), 
\href{http://dx.doi.org/10.1016/S0001-8708(03)00013-6}{513--550}.
%
\bibitem{GASUR}
T. Kobayashi, B. \O rsted, and M. Pevzner, Geometric analysis on small 
unitary representations of $GL(n,\mathbb{R})$, 
\textit{J. Funct. Anal.}, \textbf{260} (2011), 
\href{http://dx.doi.org/10.1016/j.jfa.2010.12.008}{1682--1720}.
%
\bibitem{classAq}
T. Kobayashi and Y. Oshima,
Classification of discretely decomposable 
$A_{\mathfrak {q}}(\lambda)$
 with respect to reductive symmetric pairs,
 submitted, \href{http://arxiv.org/abs/1104.4400}{arXiv:1104.4400}
%
\bibitem{xkrst}
B. Kr\"otz and R. J. Stanton,
Holomorphic extensions of representations. I. Automorphic functions,
Ann. of Math. (2) {\textbf{159}} (2004),  641--724.
%
\bibitem{NOT}
K. Nishiyama, H. Ochiai, and K. Taniguchi,
Bernstein degree and associated cycles of Harish-Chandra modules ---
Hermitian symmetric case ---,
\textit{Ast\'{e}risque}, 
\textbf{273} (2001), 13--80.
%
\bibitem{xspeh}
B. \O rsted and B. Speh,
Branching laws for some unitary representations
 of $SL(4,{\mathbb{R}})$,
 SIGMA 4 (2008)
\href{http://dx.doi.org/10.3842/SIGMA.2008.017}{doi:10.3842/SIGMA.2008.017}.  
%
\bibitem{xschthe}
W. Schmid,
Die Randwerte holomorphe Funktionen
 auf hermetisch symmetrischen Raumen.  
Invent. Math.
{\textbf{9}}
(1969--70),
61--80.  
%
\bibitem{spen}
H. Sekiguchi,
The Penrose transform for $Sp(n,\mathbb{R})$ and singular unitary
representations,
J. Math. Soc. Japan \textbf{54} (2002), 215--253.
%
\bibitem{xvoganass}
     D. A. Vogan, Jr.,
     Associated varieties and unipotent representations,
     Harmonic Analysis on Reductive Lie Groups,
\textit{Progress in Math.},
     \textbf{101} (1991),      Birkh\" auser,
       315--388.
%
\bibitem{xvoza}
D. A. Vogan, Jr. and G. J. Zuckerman,  
Unitary representations with nonzero cohomology, 
Compositio Math. \textbf{53} (1984), 51--90.
%
\bibitem{xwong}
H. Wong,
Dolbeault cohomological realization of 
Zuckerman modules associated with finite rank representations,
J. Funct. Anal. \textbf{129} (1995), 428--454.
%
\bibitem{xzhuliang}
F. Zhu and K. Liang,
On a branching law of unitary representations
 and a conjecture of Kobayashi,
C. R. Acad. Sci. Paris, Ser. I, 
\textbf{348}
(2010),
959--962.  
\end{thebibliography}
\end{document}